\documentclass[11pt,twoside]{amsart}
\usepackage{url}
\usepackage{amssymb,amsthm,amsfonts,amstext,amsmath}
\usepackage{newcent}       
\usepackage{helvet}         
\usepackage{courier}        
\usepackage{graphicx}
\usepackage{enumerate}\usepackage{mathrsfs}
\usepackage[charter]{mathdesign}
\usepackage[colorlinks=true,linkcolor=blue,citecolor=magenta]{hyperref}
\usepackage{bbm}
\usepackage{a4wide}
\usepackage{geometry}
\usepackage{nicematrix}

\usepackage{multicol}
\numberwithin{equation}{section}
\usepackage{color}
\usepackage{pstricks}
\usepackage{float}
\usepackage{mathrsfs}
\usepackage{tikz}
\usepackage{comment}

\newtheorem{theorem}{Theorem}[section]
\newtheorem{lemma}[theorem]{Lemma}
\newtheorem{proposition}[theorem]{Proposition}

\newtheorem{corollary}[theorem]{Corollary}
\newtheorem{remark}[theorem]{Remark}
\newtheorem{definition}{Definition}

\newcommand{\<}{\langle}
\renewcommand{\>}{\rangle}

\usepackage{tikz}
\usepackage{tikz-3dplot}
\usetikzlibrary{calc, shapes, backgrounds,arrows.meta,patterns}
\usetikzlibrary{decorations.markings,positioning}

\def\centerarc[#1](#2)(#3:#4:#5){\draw[#1] ($(#2)+({#5*cos(#3)},{#5*sin(#3)})$) arc (#3:#4:#5);}

\let\oldtocsection=\tocsection
\let\oldtocsubsection=\tocsubsection
\let\oldtocsubsubsection=\tocsubsubsection
\renewcommand{\tocsection}[2]{\hspace{0em}\oldtocsection{#1}{#2}}
\renewcommand{\tocsubsection}[2]{\hspace{1em}\oldtocsubsection{#1}{#2}}
\renewcommand{\tocsubsubsection}[2]{\hspace{2em}\oldtocsubsubsection{#1}{#2}}
\DeclareRobustCommand{\SkipTocEntry}[5]{}

\definecolor{deeppink}{rgb}{1.0, 0.08, 0.58}

\usepackage{hyperref}
\hypersetup{
	colorlinks=true,
	linkcolor=blue,
	filecolor=magenta,      
	urlcolor=magenta,
}

\begin{document}
	
\title[Cutoff profile for the coupon collector with many collections]{Cutoff profile for the coupon collector with many collections}

\author{M. E. Castro}
\address[M. E. Castro]{Undergraduate Program in Statistics, Universidade Federal de Santa Maria, Av. Roraima, 1000, 97105-900, Camobi, Santa Maria, Rio Grande do Sul, Brasil}
\email{maria-eliza.castro@acad.ufsm.br }

\author{R. Marinho}
\address[R. Marinho]{Department of Statistics, Universidade Federal de Santa Maria, Av. Roraima, 1000, 97105-900, Camobi, Santa Maria, Rio Grande do Sul, Brasil}
\email{rodrigo.marinho@ufsm.br}
\urladdr{\url{https://marinhor.weebly.com}}

\begin{abstract}
We consider the problem posed by Erd\H{o}s and Rényi in $1961$ about the limiting distribution of the number of draws $T_{n,m}$ that a coupon collector needs to complete $m$ collections of $n$ distinct types of coupons, when $m$ increases  with $n$.
Although the limiting distribution was studied in earlier
works, we give a simple and self-contained proof that relates it to
the cutoff phenomenon for an absorbing Markov chain.
We also provide an alternative expression for the cutoff
window in terms of the Lambert $W$ function. 
\end{abstract}


\maketitle

\tableofcontents

\section{Introduction}\label{introduction}

The coupon collector problem dates back at least to Laplace, see his seminal book of 1820~\cite{laplace}. In its classical formulation, there are $n$ distinct types of coupons, and at each step a coupon is drawn independently and uniformly at random from these $n$ types. The collector keeps all the coupons obtained, including repetitions, and the goal is to determine how long it takes until they have collected at least one coupon of each type. Nowadays this problem is a classical example in many introductory probability books and sometimes is modeled as an absorbing Markov chain, see~\cite{peres} for instance.

In 1960, Newman and Shepp mentioned that some time before Walter Weissblum proposed a natural extension of the classical problem: instead of asking how long it takes to collect one complete set of coupons, he asked how long it takes to obtain two complete sets. Observe that for two complete collections, the problem becomes more delicate. It is no longer enough to record whether each coupon type has appeared, since one must keep track of how many times each type has been observed. In~\cite{newman}, Newman and Shepp generalized this small extension to $m$ complete sets and they proved that the number $T_{n,m}$ of coupons that must be drawn so that the collector completes $m$ sets satisfies
\begin{flalign}\label{expectation}
	\mathbb E[T_{n,m}] = n \Big( \log n + (m-1)  \log \log n + C_m + o(1) \Big) ,
\end{flalign}
where $C_m$ is a constant, dependent on $m$, unknown at that moment.
One year later, in~\cite{erdos}, Erdős and Rényi analyzed the same problem and proved that
\begin{flalign}\label{tail}
	\lim_{n \to \infty}\mathbb P \Big( T_{n,m} \leq \left\lfloor  n\,\big( \log n + (m-1) \log \log n - \log (m-1)! +\theta  \big) \right\rfloor   \Big) =  \mathcal G(\theta),\,\,\forall \theta \in \mathbb R,
\end{flalign}
where $\mathcal G$ is the standard Gumbel probability distribution function, that is,
\begin{flalign}\label{gumbel}
	\mathcal G(\theta)&=e^{-e^{-\theta}},\,\,\forall \theta \in \mathbb R.
\end{flalign}
Their result also shows that the constant $C_m$ in \eqref{expectation} equals $\gamma-\log (m-1)!$ where $\gamma$ is the Euler-Mascheroni constant defined as
\begin{flalign}\label{euler}
	\gamma=\int_0^1 \frac{1-e^{-t}}{t} dt - \int_1^{\infty} \frac{e^{-t}}{t} dt = 0.5772\ldots\,.
\end{flalign}

The above results are valid only when $m$ is fixed, and at the end of their paper, Erdős and Rényi remark that ``It is an interesting problem to investigate the limiting distribution of $T_{n,m}$ when $m$ increases together with $n$, but we can not go into this question here.'' 
The problem of growing $m(n)$ was subsequently
studied by Ivchenko in~\cite{ivchenko1} and~\cite{ivchenko2}, where he obtained limiting distributions using results on occupancy problems and estimates for Gamma distribution quantiles. Papanicolaou and Doumas revisited this problem using Poissonization, and they also obtained asymptotic results for its moments and discussed an application to wireless communication~\cite{papanicolaou}. More recently, Long also stated a Gumbel limit for arbitrary sequences $m(n)$, together with asymptotic results for the expectation and variance, see~\cite{long}.

The coupon collector problem is interesting not only because of its history. As mentioned, the problem also has an interpretation in wireless communication, see~\cite{papanicolaou,sharif1} and~\cite{sharif2} for this connection. Moreover, if we put some attention in (\ref{tail}) we can see that it asserts that, when $m$ is fixed, as $n\to \infty$, the probability $\mathbb P (T_{n,m}> \ell)$ converges abruptly from $1$ to $0$ at $\ell=n\,\big( \log n + (m-1) \log \log n - \log (m-1)!\big)$, a phenomenon called \textit{cutoff} in the literature. This result also shows that this abrupt convergence happens within a window of order $n$ and presents a standard Gumbel profile.

In this work we describe the limiting distribution in terms of cutoff for an absorbing Markov chain. Let $F_m$ be the distribution function of a $\text{Gamma}(m,1)$ random variable.
We define the cutoff time by $t_{n,m}=nF_m^{-1}(1-1/n)$, which is also the choice made by Ivchenko in~\cite{ivchenko2} and by Long in~\cite{long}, and we obtain an explicit expression for the cutoff window in terms of the real branch $W_{-1}$ of the Lambert $W$ function, not given in these earlier works. Our contribution concerns the simple and self-contained proof, its connection with the cutoff phenomenon, and the alternative expression for the cutoff window. We obtain the required estimates directly from finite sums and prove the expectation expansion by controlling both tails of the absorption time.

Another contribution of this work is the definition of poissonization for general discrete-time stochastic processes and the attainment of a result that we call \textit{Depoissonization Lemma}. In the literature, poissonization is the general procedure of replacing a deterministic discrete parameter by an independent Poisson random variable, and it has been applied to the analysis of many variations of the coupon collector problem (see~\cite{holst,ilienko} and~\cite{papanicolaou} for instance). In this work, given a discrete-time stochastic process $\{X_\ell ; \ell \in \mathbb N_0\}$ and a Poisson process with intensity $1$, $\{ N(t) ; t \geq 0 \}$, independent of $\{X_\ell ; \ell \in \mathbb N_0\}$, we say that a stochastic process $\{ Y_t; t \geq 0\}$ is the poissonization of $\{X_\ell ; \ell \in \mathbb N_0\}$ if it is defined as $Y_t= X_{N(t)}$ for every $t \geq 0$. In particular, if $\{X_\ell ; \ell \in \mathbb N_0\}$ is a finite discrete-time Markov chain with transition matrix $P$, then its poissonization is a continuous-time Markov chain with generator $P -I$ and it should preserve the expectation of a hitting time.
The so-called Depoissonization Lemma allows us to analyze the limiting distribution of standardized hitting times of discrete-time stochastic processes in the continuous-time setting by using its poissonization. The combination of this procedure with the Depoissonization Lemma brings a new methodology with it, that we use to analyze the coupon collector problem with many collections. We believe that this method can be applied to many other problems. 

The paper is organized as follows: In Section~\ref{def and res} we construct a discrete-time absorbing Markov chain which is absorbed at time $T_{n,m}$, we present the definitions of cutoff, cutoff window and cutoff profile, and we enunciate the result about the limiting distribution of $T_{n,m}$, see Theorem~\ref{main result}. In Section~\ref{poissonization} we define the poissonization of a discrete-time stochastic process and we prove that not only it preserves expectations of hitting times, but it can also be used, through the Depoissonization Lemma, to analyze the limiting distributions of such stopping times in the continuous-time setting. This section may be of independent interest for the reader and can be read separately. In Section~\ref{main arg} we analyze the poissonization of the discrete-time Markov chain previously constructed and we use it, with the Depoissonization Lemma, to prove Theorem~\ref{main result}. For completeness, in Appendix~\ref{A}, we prove the equivalence between analyzing the distribution of the absorption time of an absorbing finite Markov chain with fixed initial state and analyzing the total variation distance between its law and its target stationary distribution. In Appendix~\ref{B}, we use Yau's relative entropy method~\cite{yau} to show that, for the poissonization of the Markov chain previously constructed, the number of copies of each of the $n$ types of coupons are independent and identically distributed with explicit marginals, which may not be immediate for some readers. We remark that, although these appendices are not necessarily essential to solve the proposed problem, the main ideas to analyze it came from them, being this the reason to present them separately.

\section{Definitions and main result}\label{def and res}

\subsection{A discrete-time Markov model}

We consider the coupon collector problem described in the introduction: there are $n\geq 2$ distinct types of coupons, and at each step a coupon is drawn independently and uniformly at random from these $n$ types. The collector keeps all the coupons obtained, including repetitions, and the focus is on the number $T_{n,m}$ of needed draws so that the collector has at least $m$ coupons of each type. To investigate the distribution of $T_{n,m}$ we model this problem as the discrete time absorbing Markov chain $\{X_\ell^{n,m}; \ell \in \mathbb N_0 \}$\footnote{$\mathbb N_0:=\mathbb N \cup \{0\}$.} with state space $\Omega_{n,m}=\{0,1,\ldots,m\}^n$ and transition matrix $P$ given by
\begin{flalign*}
	P (x, y) = \begin{cases}
		\frac{\mathbbm 1_{\{ x_i < m \}}}{n} ,& \text{ if } y = x^{i,+} \text{ for some } i \in \{1,\ldots,n\},\\
	 \frac{1}{n} \sum_{j =1}^{n} \mathbbm 1_{\{ x_j = m \}}   , &\text{ if } y = x,\\
	 0 & \text{ otherwise},
	\end{cases}
\end{flalign*}
where for every $i \in \{1,2,\ldots,n\}$, $x_i$ denotes the $i$-th coordinate of a vector $x \in \Omega_{n,m}$ and $x^{i,+}\in \mathbb R^n$ is the vector obtained from $x$ after adding $1$ at the $i$-th coordinate:
\begin{flalign*}
	x^{i,+}_k=\begin{cases}
		x_k + 1 &, \text{ if } k=i,\\
		x_k &, \text{ if } k \in \{1,2,\ldots,n\}\setminus\{i\}.
	\end{cases}
\end{flalign*}
Therefore, the $i$-th coordinate of the vector $X_\ell^{n,m}$ stands for the minimum between $m$ and the number of copies of type $i$ at time $\ell \in \mathbb N_0$.

We start the chain from $(0,\ldots,0)\in \Omega_{n,m}$, since it represents the state for when the collector has no coupons. Clearly, the stationary measure $\pi^{n,m}$ of this absorbing Markov chain is a Dirac measure concentrated in the unique absorbing state $(m,\ldots,m) \in \Omega_{n,m}$ and 
\begin{flalign}\label{T}
	T_{n,m}:=\min \left\{ \ell \in \mathbb N_0;\, X_\ell^{n,m} = (m,\ldots,m) \right\}.
\end{flalign}

\subsection{Cutoff phenomenon}

Below we define the cutoff phenomenon in the context of absorbing Markov chains.

\begin{definition}
Let $n$ be a scaling parameter. Let $\{ X_\ell^n ; \ell \in \mathbb N_0  \}$ be a discrete-time absorbing Markov chain on a finite state space $\Omega_n$ with fixed initial state $x_0^n \in \Omega_n$. Let $\mathcal A_n$ be the set of all absorbing states that can be reached from $x_0$ in a finite number of steps. Let $T_n=\min\{ \ell \in \mathbb N_0;  X_\ell^n \in \mathcal A_n  \}$ be the absorption time of the process. We say that this Markov chain presents cutoff at times $t_n>0$ with a window of size $w_n>0$ if $w_n=o(t_n)$ and if for any $\varepsilon >0$ there exists $\theta>0$ such that $\limsup_{n\to \infty} \mathbb P( T_n > t_n + \theta w_n) \leq \varepsilon$ and $\liminf_{n\to \infty} \mathbb P(T_n > t_n - \theta w_n) \geq 1-\varepsilon$. Moreover, if for any $\theta \in \mathbb R$ the probability $\mathbb P(T_n \leq t_n + \theta w_n)$ converges to $ G (\theta)$ for some probability distribution function $ G: \mathbb R \to [0,1]$, then we say that $G$ is the cutoff profile of this Markov chain. In particular,
\begin{flalign*}
\frac{T_n - t_n}{w_n} \stackrel{n \to \infty}{\longrightarrow} X,
\end{flalign*}
in distribution, where $X$ is a random variable with probability distribution function $G$.
\end{definition}

\begin{remark}
In order to justify the definition of cutoff for absorbing Markov chains given above, denote by $\mu_\ell^n$ the law of $X_\ell^n$ and by $\pi^n$ the target stationary measure of $\{ X_\ell^n ; \ell \in \mathbb N_0  \}$ for which $\mu_\ell^n \to \pi^n$ as $\ell \to \infty$. Proposition~\ref{dtv} asserts that
\begin{flalign*}
\| \mu_\ell^n - \pi^n \|_{\text{TV}}  &=\mathbb P(T_n > \ell).
\end{flalign*}
Therefore, the definition of cutoff given above is equivalent to total variation cutoff for absorbing Markov chains. 
\end{remark}

\begin{remark}
The definitions of cutoff, cutoff window and cutoff profile can be given for continuous-time Markov chains replacing $\mathbb N_0$ by $[0,\infty)$ above.
\end{remark}

\subsection{Main result}

Our main result gives the cutoff time, window, and profile for the coupon collector with many collections. It also gives an expansion of the expectation with an error smaller than the cutoff window.

\begin{theorem}\label{main result}
Let $F_m$ be the probability distribution function of a Gamma$(m,1)$, that is, 
	\begin{flalign*}
		F_m(t)= \begin{cases}
		1 - e^{-t} \sum_{k=0}^{m-1} \frac{t^k}{k!} &,\text{ if } t > 0,\\
		0 &, \text{ if } t\leq 0.
		\end{cases}
	\end{flalign*}
Define
\begin{flalign}\label{tn}
t_{n,m}:=n \big(F_{m}^{-1}(1-1/n) \big).
\end{flalign}
Let $W_{-1}$ denote the real branch of the Lambert $W$ function defined on
$[-e^{-1},0)$ by
\begin{flalign*}
W_{-1}(z)e^{W_{-1}(z)}=z,
\qquad W_{-1}(z)\leq -1.
\end{flalign*}
Define
\begin{flalign}\label{wn}
w_{n,m}:=\begin{cases}
n &, \text{ if } \lim_{n \to \infty}\frac{m}{\log n} = 0,\\
\frac{nW_{-1}(-e^{-1-1/\alpha})}{W_{-1}(-e^{-1-1/\alpha}) + 1} &, \text{ if } \lim_{n \to \infty}\frac{m}{\log n} = \alpha \in (0,+\infty),\\
\frac{nW_{-1}(-e^{-1-\log n/m})}{W_{-1}(-e^{-1-\log n/m})+ 1} &, \text{ if } \lim_{n \to \infty}\frac{m}{\log n} =+\infty.
\end{cases}
\end{flalign}
Then, for each $\theta \in \mathbb R$, we have
	\begin{flalign}\label{erdos}
		\lim_{n \to \infty} \mathbb P \left( T_{n,m} \leq \left\lfloor  t_{n,m} + \theta w_{n,m} \right\rfloor \right)=\mathcal G(\theta),
	\end{flalign}
where $\mathcal G$ is the standard Gumbel probability distribution function given in (\ref{gumbel}). Furthermore,
	\begin{flalign}\label{expect}
		\mathbb E [T_{n,m}]= t_{n,m} +  \gamma w_{n,m} + o\left(w_{n,m}\right) ,
	\end{flalign}
	where $\gamma$ is the Euler–Mascheroni constant given in (\ref{euler}).
\end{theorem}

\begin{remark}
The reader can easily check that when $m$ is fixed, we have $t_{n,m}= n \big(\log n +(m-1)  \log \log n - \log (m-1)! +o(1) \big)$. Moreover, along the proofs, it will be possible to verify that, in this case, $t_{n,m}$ can be replaced by $ n \big(\log n +(m-1)  \log \log n - \log (m-1)! \big)$. Since our focus is on the regime in which $m$ grows with $n$, we prefer to enunciate Theorem~\ref{main result} in the presented form. Furthermore, the three expressions for the cutoff window above are
asymptotically equivalent to
$$
n\frac{W_{-1}(-e^{-1-(\log n)/m})}
{W_{-1}(-e^{-1-(\log n)/m})+1}.
$$
We present the proof separately in each regime to introduce the required estimates progressively.
\end{remark}

\section{Poissonization}\label{poissonization}

In this section we present a methodology, inspired in a technique called \textit{Poissonization}, to analyze tail probabilities of hitting times of discrete-time stochastic processes. This methodology will be used, in the next section, to prove Theorem~\ref{main result}. As mentioned in the introduction, poissonization is the general procedure of replacing a deterministic discrete parameter by an independent Poisson random variable. In what follows we define this notion for discrete-time stochastic processes:
\begin{definition}
Let $\{X_\ell ; \ell \in \mathbb N_0\}$ be a discrete-time stochastic process. Let $\{ N(t) ; t \geq 0 \}$ be a Poisson process with intensity $1$, independent of $\{X_\ell ; \ell \in \mathbb N_0\}$. For each $t \geq 0$ define
	\begin{flalign*}
		Y_t := X_{N(t)}.
	\end{flalign*} 
	The process $\{ Y_t; t \geq 0\}$ is called the \textit{poissonization} of $\{X_\ell ; \ell \in \mathbb N_0\}$.
\end{definition}

\begin{remark}
The intensity of the Poisson process used in the previous definition is properly chosen so that if $\{X_\ell ; \ell \in \mathbb N_0\}$ is a discrete-time homogeneous Markov chain with transition matrix $P$, then its poissonization is a continuous-time Markov chain with generator $P -I$.
\end{remark}

An important property that we gain from a poissonization of a discrete-time stochastic process is that it does not change the expectation of a hitting time. Although it may be clear to many readers, we prove it below:

\begin{proposition}\label{expectations}
Let $\{ X_\ell ; \ell \in \mathbb N_0  \}$ be a discrete-time stochastic process with state space $\Omega$. Let $\mathcal A$ be a subset of $\Omega$. Let $T=\min\{ \ell \in \mathbb N_0;  X_\ell \in \mathcal A  \}$.
Let $\{ Y_t ; t \geq 0  \}$ be the Poissonization of $\{ X_\ell ; \ell \in \mathbb N_0  \}$. Let $\tau:=\inf \{ t \geq 0 ; Y_t \ \in \mathcal A  \}$. We have
	\begin{flalign*}
		\mathbb E[T]=\mathbb E[\tau].
	\end{flalign*}
\end{proposition}

\begin{proof}
	For each $k \in \mathbb N_0$ define $\mathcal E_k:=\inf\{  t \geq 0 ; N(t) \geq k \}$. Observe that $\mathcal E_k$ has distribution Gamma$(k,1)$ because it is the sum of $k$ independent exponential random variables with rate $1$. Moreover, $\mathcal E_{T}= \tau$. Therefore,
	\begin{flalign*}
		\mathbb E[\tau] &= \mathbb E [\mathcal E_{T}]= \sum_{k=0}^{+\infty} \mathbb E [\mathcal E_{T} | T = k] \, \mathbb P (T=k)= \sum_{k=0}^{+\infty} \mathbb E [\mathcal E_{k} ] \, \mathbb P (T=k)= \sum_{k=0}^{+\infty} k \, \mathbb P (T=k)=\mathbb E[T].
	\end{flalign*}
\end{proof}

Observe that the above result allows us to obtain the expectation of a discrete-time hitting time from its continuous-time counterpart. The next lemma gives a sufficient condition under which the corresponding limiting distributions also agree.

\begin{lemma}[Depoissonization Lemma]\label{dePoissonization}
	Let $n$ be a scaling parameter. Let $\{ X_\ell^n ; \ell \in \mathbb N_0  \}$ be a discrete-time stochastic process with state space $\Omega_n$. Let $\mathcal A_n$ be a subset of $\Omega_n$. Let $T_n=\min\{ \ell \in \mathbb N_0;  X_\ell^n \in \mathcal A_n  \}$.
	Let $\{ Y_t^n ; t \geq 0  \}$ be the Poissonization of $\{ X_\ell^n ; \ell \in \mathbb N_0  \}$. Let $\tau_n:=\inf \{ t \geq 0 ; Y_t^n \ \in \mathcal A_n  \}$.
	Assume that there exist times $t_n,w_n \geq 0$ and a continuous probability distribution function $ G: \mathbb R \to [0,1]$ such that $\lim_{n \to \infty} \mathbb P \Big(\tau_n \leq t_n + \theta w_n\Big)= G(\theta)$ for every $\theta \in \mathbb R$. If $t_n=o\big(|w_n|^2\big)$ and $w_n=o(t_n)$ then
	\begin{flalign*}
		\lim_{n \to \infty}	\mathbb P \Big( T_{n} \leq \left\lfloor t_{n} + \theta w_n\right\rfloor  \Big) =  G (\theta ).
	\end{flalign*}
\end{lemma}

\begin{proof}
Observe that $\{\tau_n \leq t\}=\{T_n\leq N(t)\}$, for every $t\geq0$. Fix $\theta\in\mathbb R$ and $\varepsilon>0$. Since
	$w_n=o(t_n)$, both $t_n+(\theta+\varepsilon)w_n$ and
	$t_n+(\theta-\varepsilon)w_n$ are positive for sufficiently
	large $n$. Moreover, since $t_n=o\big(|w_n|^2\big)$ and $w_n=o(t_n)$, we have
\begin{flalign}\label{w}
	\frac{1}{w_n} =\frac{w_n}{t_n}\frac{t_n}{|w_n|^2}=o(1).
\end{flalign}

First, we obtain an upper bound for $\mathbb P\big(T_n\leq\lfloor t_n+\theta w_n\rfloor\big)$:
\begin{flalign*}
\mathbb P\big(T_n\leq\lfloor t_n+\theta w_n\rfloor\big)&=\mathbb P\Big(
		T_n\leq\lfloor t_n+\theta w_n\rfloor,\,
		N\big(t_n+(\theta+\varepsilon)w_n\big)
		\geq\lfloor t_n+\theta w_n\rfloor
		\Big)\\
		&\quad+\mathbb P\Big(
		T_n\leq\lfloor t_n+\theta w_n\rfloor,\,
		N\big(t_n+(\theta+\varepsilon)w_n\big)
		<\lfloor t_n+\theta w_n\rfloor
		\Big)\\
		&\leq\mathbb P\Big(
		T_n\leq N\big(t_n+(\theta+\varepsilon)w_n\big)
		\Big)
		+\mathbb P\Big(
		N\big(t_n+(\theta+\varepsilon)w_n\big)
		<\lfloor t_n+\theta w_n\rfloor
		\Big)\\
		&=\mathbb P\big(\tau_n\leq t_n+(\theta+\varepsilon)w_n\big)
		+\mathbb P\Big(
		N\big(t_n+(\theta+\varepsilon)w_n\big)
		<\lfloor t_n+\theta w_n\rfloor
		\Big).
	\end{flalign*}
	Since $\lfloor t_n+\theta w_n\rfloor\leq t_n+\theta w_n$ and $t_n=o\big(|w_n|^2\big)$, by
	Tchebyshev's inequality and (\ref{w}), we have
	\begin{flalign*}
		&\mathbb P\Big(
		N\big(t_n+(\theta+\varepsilon)w_n\big)
		<\lfloor t_n+\theta w_n\rfloor
		\Big)\\
		&\leq\mathbb P\Big(
		N\big(t_n+(\theta+\varepsilon)w_n\big)
		-\big(t_n+(\theta+\varepsilon)w_n\big)
		<-\varepsilon w_n
		\Big)\\
		&\leq\mathbb P\Big(
		\big|
		N\big(t_n+(\theta+\varepsilon)w_n\big)
		-\big(t_n+(\theta+\varepsilon)w_n\big)
		\big|>\varepsilon w_n
		\Big)\\
		&\leq\frac{t_n+(\theta+\varepsilon)w_n}
		{\varepsilon^2|w_n|^2}\\
		&=\frac{1}{\varepsilon^2} \frac{t_n}{|w_n|^2} + \frac{\theta+\varepsilon}{\varepsilon^2} \frac{1}{w_n}=o(1).
	\end{flalign*}
Consequently,
	\begin{flalign*}
		\limsup_{n\to\infty}
		\mathbb P\big(T_n\leq\lfloor t_n+\theta w_n\rfloor\big)
		&\leq
		\lim_{n\to\infty}
		\mathbb P\big(\tau_n\leq t_n+(\theta+\varepsilon)w_n\big)= G(\theta+\varepsilon).
	\end{flalign*}
	
	For the lower bound, observe that
	\begin{flalign*}
		&\mathbb P\big(T_n\leq\lfloor t_n+\theta w_n\rfloor\big)\\
		&\geq\mathbb P\Big(
		T_n\leq N\big(t_n+(\theta-\varepsilon)w_n\big),\,
		N\big(t_n+(\theta-\varepsilon)w_n\big)
		\leq\lfloor t_n+\theta w_n\rfloor
		\Big)\\
		&\geq\mathbb P\Big(
		T_n\leq N\big(t_n+(\theta-\varepsilon)w_n\big)
		\Big)
		-\mathbb P\Big(
		N\big(t_n+(\theta-\varepsilon)w_n\big)
		>\lfloor t_n+\theta w_n\rfloor
		\Big)\\
		&=\mathbb P\big(\tau_n\leq t_n+(\theta-\varepsilon)w_n\big)
		-\mathbb P\Big(
		N\big(t_n+(\theta-\varepsilon)w_n\big)
		>\lfloor t_n+\theta w_n\rfloor
		\Big).
	\end{flalign*}
	Since the Poisson process takes integer values, we have
	\begin{flalign*}
\left\{  	N\big(t_n+(\theta-\varepsilon)w_n\big)
>\lfloor t_n+\theta w_n\rfloor   \right\} \subset \left\{   	N\big(t_n+(\theta-\varepsilon)w_n\big)>t_n+\theta w_n   \right\}.
	\end{flalign*}
	Therefore, again, by Tchebyshev's inequality and (\ref{w}), and since $t_n=o\big(|w_n|^2\big)$, we obtain
	\begin{flalign*}
		&\mathbb P\Big(
		N\big(t_n+(\theta-\varepsilon)w_n\big)
		>\lfloor t_n+\theta w_n\rfloor
		\Big)\\
		&\leq\mathbb P\Big(
		N\big(t_n+(\theta-\varepsilon)w_n\big)
		-\big(t_n+(\theta-\varepsilon)w_n\big)
		>\varepsilon w_n
		\Big)\\
		&\leq\frac{t_n+(\theta-\varepsilon)w_n}
		{\varepsilon^2w_n^2}\\
	&=\frac{1}{\varepsilon^2} \frac{t_n}{|w_n|^2} + \frac{\theta-\varepsilon}{\varepsilon^2} \frac{1}{w_n}=o(1).
	\end{flalign*}
	Hence,
	\begin{flalign*}
		\liminf_{n\to\infty}
		\mathbb P\big(T_n\leq\lfloor t_n+\theta w_n\rfloor\big)
		&\geq
		\lim_{n\to\infty}
		\mathbb P\big(\tau_n\leq t_n+(\theta-\varepsilon)w_n\big)= G(\theta-\varepsilon).
	\end{flalign*}
	Combining the two bounds, we conclude that
	\begin{flalign*}
		 G(\theta-\varepsilon)
		&\leq\liminf_{n\to\infty}
		\mathbb P\big(T_n\leq\lfloor t_n+\theta w_n\rfloor\big)\leq\limsup_{n\to\infty}
		\mathbb P\big(T_n\leq\lfloor t_n+\theta w_n\rfloor\big)
		\leq G(\theta+\varepsilon).
	\end{flalign*}
	Letting $\varepsilon\downarrow0$ and using the continuity
	of $\mathcal G$ at $\theta$, we obtain
	\[
	\lim_{n\to\infty}
	\mathbb P\big(T_n\leq\lfloor t_n+\theta w_n\rfloor\big)
	= G(\theta).
	\]
\end{proof}

\section{Proof of the main result}\label{main arg}

In this section we prove Theorem~\ref{main result}. For that purpose, let $\{ Y_t^{n,m}; t \geq 0\}$ be the poissonization of the coupon collector Markov chain. By definition, it is a continuous-time absorbing Markov chain with state space $\Omega_{n,m}$ and infinitesimal generator given by
\begin{flalign*}
	\mathfrak L_{n,m} f({ x})=\frac{1}{n} \sum_{i=1}^n \mathbbm 1_{\{ x_i < m \}} \left( f({x}^{i,+})- f({ x}) \right),\,\,\forall x \in \Omega_{n,m}
\end{flalign*}
for every $f:\Omega_{n,m} \to \mathbb R$. Define 
\begin{flalign}\label{tau}
\tau_{n,m}:=\min\{ t \geq 0 ; Y_t^{n,m}=(m,\ldots,m) \}.
\end{flalign}
Recall the definitions of $\mathcal G$, $\gamma$, $t_{n,m}$ and $w_{n,m}$ given in (\ref{gumbel}), (\ref{euler}), (\ref{tn}) and (\ref{wn}), respectively. The idea is to show that $\mathbb P (\tau_{n,m} \leq t_{n,m} + \theta w_{n,m})$ converges to $\mathcal G (\theta)$ for every $\theta \in \mathbb R$, as $n \to \infty$ and then to combine it with Proposition~\ref{expectations} and the Depoissonization Lemma (Lemma~\ref{dePoissonization}) to conclude Theorem~\ref{main result}.

We first observe that, by the \textit{thinning property} of Poisson processes, if one looks at the $k$-th coordinate of the Markov chain $\{ Y_t^{n,m}; t \geq 0\}$, say $\{ Y_t^{n,m}(k); t \geq 0\}$, then they can observe that it follows the law of a Poisson process with intensity $1/n$ until it is absorbed at $m$ and it is independent of the other coordinates values. 
Indeed, in Proposition~\ref{thinning} we prove that for every $t\geq 0$ the random variables $Y_t^{n,m}(1),\ldots, Y_t^{n,m}(n)$ are independent and identically distributed with probability mass function $p_t:\{0,1,\ldots,m\}\to [0,1]$ given by
\begin{flalign}\label{pt}
p_t(k)=  \mathbbm 1_{\{ k <m   \}}\,e^{-t/n} \frac{(t/n)^{k}}{k!}   + \mathbbm 1_{\{ k =m   \}} \left( 1 - \,e^{-t/n} \sum_{j=0}^{m-1}\frac{(t/n)^{j}}{j!} \right),\,\,\forall k \in \{0,1,\ldots,m\}.
\end{flalign}
Recall the definition of $F_m$ given in the statement of Theorem~\ref{main result}. From the above observation we conclude that the distribution of $\tau_{n,m}/n$ is given by the maximum among $n$ independent Gamma$(m,1)$ random variables:
\begin{flalign}\label{tau dist}
\mathbb P (\tau_{n,m} \leq t) =\mathbb P \big(Y_t^{n,m}=(m,\ldots,m)\big)= \left( 1 - \,e^{-t/n} \sum_{j=0}^{m-1}\frac{(t/n)^{j}}{j!} \right)^n= \big(F_m(t/n)\big)^n,\, \forall t \geq 0.
\end{flalign}
The graph of the function $t \mapsto \mathbb{P}(\tau_{n,m}\leq t)$ illustrates the abrupt convergence from $0$ to $1$, see the example in which $n=980$ and $m=100$ depicted in Figure~\ref{plot}:
\begin{figure}[H]
	\centering
	\includegraphics[scale=0.4]{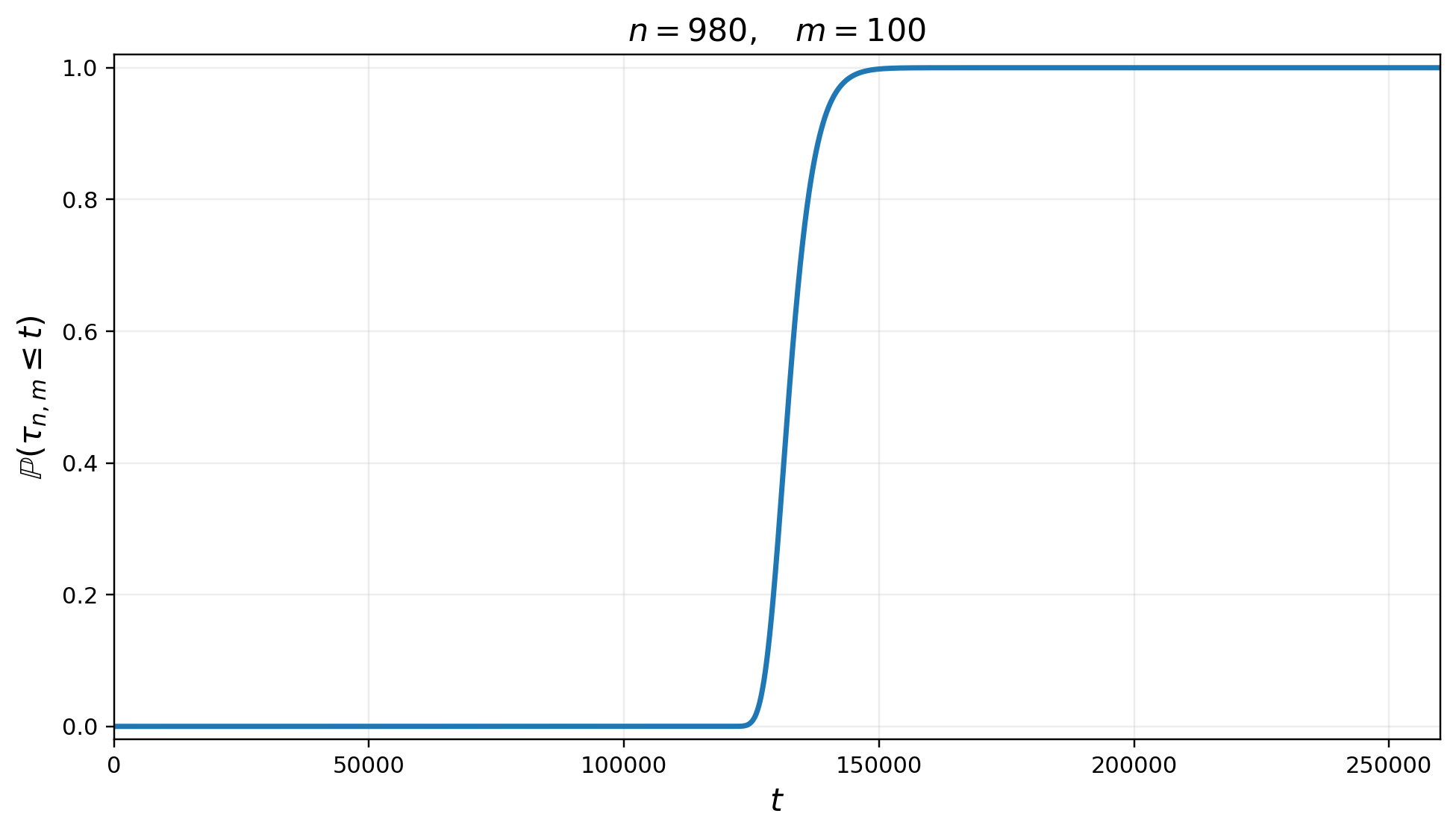}
	\caption{Graph of $t \mapsto \mathbb{P}(\tau_{n,m}\leq t)$, with $n=980$ and $m=100$.}
	\label{plot}
\end{figure}

To accomplish our goal, for each $\theta \in \mathbb R$ we must determine $\tilde t_{n,m}(\theta)$ such that $\Big(F_m\big(\tilde t_{n,m}(\theta)/n\big)\Big)^n$ converges to $\mathcal G(\theta)=e^{-e^{-\theta}}$. Clearly, this happens when
\begin{flalign*}
F_m\big(\tilde t_{n,m}(\theta)/n\big)=1 - \frac{e^{-\theta}+o(1)}{n},
\end{flalign*}
and we must have $\tilde t_{n,m}(\theta)=n F_m^{-1} \left(1 - \frac{e^{-\theta}+o(1)}{n} \right)$. 
We choose $t_{n,m}$ as in (\ref{tn}). This choice fixes the location of the standard Gumbel distribution in the limiting formula. 

The following lemma shows that the sequence $w_{n,m}$, defined in (\ref{wn}), satisfies $w_{n,m}=o(t_{n,m})$ and $t_{n,m}=o(|w_{n,m}|^2)$, and that we can replace $\tilde t_{n,m}(\theta)$ by $t_{n,m}+\theta w_{n,m}$.

\begin{lemma}\label{profile}
	The following statements hold:
\begin{enumerate}
	\item[I.] For each $\theta \in \mathbb R$, we have
	\begin{flalign*}
		F_m \left(\frac{t_{n,m} + \theta w_{n,m}}{n} \right) =1- \frac{e^{-\theta} + o(1) }{n} ;
	\end{flalign*}
	\item[II.] $w_{n,m}=o(t_{n,m})$;
	\item[III.] $t_{n,m}=o(|w_{n,m}|^2)$.
\end{enumerate}
\end{lemma}

\begin{proof}
Define $u_{n,m}:=t_{n,m}/n$. Let $\theta \in \mathbb R$. The idea is to observe the behavior of $u_{n,m}$. Thus, we split the proof in three cases, one for each regime of $m=m(n)$, which becomes more delicate as $m$ increases faster. \\ 

\noindent \textbf{Case I ($\lim_{n \to \infty} \frac{m}{\log n}=0$):} 
For this case, observe that
\begin{flalign}
\frac{1-F_m(u_{n,m}+\theta)}{1-F_m(u_{n,m})}&=\frac{e^{-(u_{n,m}+\theta)} \sum_{k=0}^{m-1} \frac{(u_{n,m}+\theta)^k}{k!}}{e^{-(u_{n,m})} \sum_{k=0}^{m-1} \frac{(u_{n,m})^k}{k!}} \label{1-F}\\&=e^{-\theta} \frac{\sum_{k=0}^{m-1} \frac{(u_{n,m})^k}{k!}\left( 1 +\frac{\theta}{u_{n,m}} \right)^k}{\sum_{k=0}^{m-1} \frac{(u_{n,m})^k}{k!}} \nonumber \\
&=e^{-\theta}  + 
 \frac{e^{-\theta} \sum_{k=0}^{m-1} \frac{(u_{n,m})^k}{k!}\left[\left( 1 +\frac{\theta}{u_{n,m}} \right)^k - 1\right]}{\sum_{k=0}^{m-1} \frac{(u_{n,m})^k}{k!}}. \nonumber 
\end{flalign}
Since $e^{-u_{n,m}} \leq 1-F_m(u_{n,m}) = \frac{1}{n}$, we have $u_{n,m} \geq \log n$. 
Therefore,
\begin{flalign*}
\lim_{n \to \infty} \max_{k \in \{0,1,\ldots,m-1\}} \left| \left( 1 +\frac{\theta}{u_{n,m}} \right)^k - 1 \right| =0.
\end{flalign*}
Hence,
\begin{flalign*}
1- F_m \left(\frac{t_{n,m} + \theta w_{n,m}}{n} \right)=1-F_m(u_{n,m}+\theta)&=(1-F_m(u_{n,m}))\,(e^{-\theta}  + o(1))=\frac{e^{-\theta}  + o(1)}{n}.
\end{flalign*}
Moreover,
\begin{flalign*}
	\frac{w_{n,m}}{t_{n,m}}=\frac{n}{n u_{n,m}} \leq \frac{1}{\log n}= o(1).
\end{flalign*}
To conclude that $t_{n,m}=o(|w_{n,m}|^2)$, let $Z_m$ be a random variable with distribution Gamma$(m,1)$. By Markov's inequality,
\begin{flalign*}
\frac{1}{n}=1- F_m(u_{n,m})= \mathbb P (Z_m > u_{n,m}) \leq \frac{\mathbb E [e^{Z_m / 2}]}{e^{u_{n,m}/2}}=2^m e^{- u_{n,m} / 2}.
\end{flalign*}
Therefore, $u_{n,m} \leq 2 \log n + ( \log 4) m $. Hence,
\begin{flalign*}
\frac{t_{n,m}}{|w_{n,m}|^2} = \frac{u_{n,m}}{n} \leq \frac{2 \log n + ( \log 4) m}{n} = o(1).
\end{flalign*}

\noindent \textbf{Case II ($\lim_{n \to \infty} \frac{m}{\log n}=\alpha\in (0,+\infty)$):} 

In this regime we first show that $u_{n,m}$ behaves as $-W_{-1}(-e^{-1-1/\alpha})m$. We start observing that, since $\lim_{n \to \infty} \frac{m}{\log n}=\alpha\in (0,+\infty)$, we have
\begin{flalign}\label{ldp u}
	\frac{1}{m}\log (1-F_m(u_{n,m})) &=\frac{1}{m} \log n^{-1}=- \frac{\log n}{m}=-\frac{1}{\alpha} +o(1).
\end{flalign}
Another observation is that if $r>1$ then
\begin{flalign*}
\frac{\frac{(rm)^{k-1}}{(k-1)!}}{\frac{(rm)^{k}}{k!}}= \frac{k}{rm} \leq \frac{m-1}{r m} < \frac{1}{r}
\end{flalign*}
for every $k \in \{1,\ldots,m-1\}$. Therefore,
\begin{flalign*}
\sum_{k=0}^{m-2}\frac{(rm)^{k}}{k!}=\sum_{k=1}^{m-1}\frac{(rm)^{k-1}}{(k-1)!} < \frac{1}{r} \sum_{k=1}^{m-1}\frac{(rm)^{k}}{k!} < \frac{1}{r} \sum_{k=0}^{m-1}\frac{(rm)^{k}}{k!}= \frac{1}{r} \left(\frac{(rm)^{m-1}}{(m-1)!} +  \sum_{k=0}^{m-2}\frac{(rm)^{k}}{k!} \right).
\end{flalign*}
Hence,
\begin{flalign*}
\left(1 - \frac{1}{r} \right) \sum_{k=0}^{m-2}\frac{(rm)^{k}}{k!} < \frac{1}{r} \frac{(rm)^{m-1}}{(m-1)!}.
\end{flalign*}
Dividing both sides above by $1-1/r$ we obtain
\begin{flalign*}
 \sum_{k=0}^{m-2}\frac{(rm)^{k}}{k!} < \frac{1}{r-1} \frac{(rm)^{m-1}}{(m-1)!}.
\end{flalign*}
Adding $\frac{(rm)^{m-1}}{(m-1)!}$ on both sides of the above inequality, we obtain
\begin{flalign*}
\frac{(rm)^{m-1}}{(m-1)!}<	\sum_{k=0}^{m-1}\frac{(rm)^{k}}{k!} < \frac{r}{r-1} \frac{(rm)^{m-1}}{(m-1)!}.
\end{flalign*}
We conclude that for any $r> 1$ we have
\begin{flalign*}
\log (1-F_m(rm)) &=- rm + (m-1) \log (rm) -\log (m-1)! + \mathcal O (1).
\end{flalign*}
Using Stirling's formula one can see that $\log (m-1)! = (m-1) \log (m-1) - (m-1) + O(\log (m-1))$.
Therefore,
\begin{flalign*}
	\log (1-F_m(rm)) &=- m I(r) +o(m),
\end{flalign*}
where $I(r):=r -1 - \log r$. Hence,
\begin{flalign}\label{epsilon}
 \lim_{n \to \infty}	\frac{1}{m}\log (1-F_m(r m)) &=-  I(r) .
\end{flalign}
Observe that 
$I'(r)=1-\frac{1}{r}>0$ for every $r>1$. Hence, $I$ is strictly increasing on $(1,+\infty)$. 
Let $w>1$ be the unique solution of $I(w) = 1/\alpha$. Then, 
\begin{flalign*}
-w+\log w = -1-1/\alpha.
\end{flalign*}
Applying the exponential to both sides of the above identity we obtain
\begin{flalign*}
-w e^{-w} = -e^{-1-1/\alpha}.
\end{flalign*}
Therefore, $w =-W_{-1}(-e^{-1-1/\alpha})=:r_{\alpha}$.

Now, fix $\delta \in (0,r_\alpha-1)$. Thus,
\begin{flalign*}
I(r_\alpha - \delta) < I(r_\alpha) = \frac{1}{\alpha} < I(r_\alpha + \delta) .
\end{flalign*}
By (\ref{ldp u}) and (\ref{epsilon}), 
we have
\begin{flalign*}
\lim_{n \to\infty} \left\{\frac{1}{m}\log \Big(1-F_m\big((r_\alpha - \delta) m\big)\Big) - \frac{1}{m}\log \Big(1-F_m\big( u_{n,m }\big)\Big)\right\}&=-  I(r_\alpha - \delta) + \frac{1}{\alpha}>0. 
\end{flalign*}
Then, for sufficiently large $n$ we have
\begin{flalign*}
F_m\big((r_\alpha - \delta) m\big) < F_m\big( u_{n,m }\big).
\end{flalign*}
Since $F_m$ is non-decreasing, we conclude that for sufficiently large $n$ we have $(r_\alpha - \delta) m \leq u_{n,m }$. Analogously, since
\begin{flalign*}
	\lim_{n \to\infty} \left\{\frac{1}{m}\log \Big(1-F_m\big((r_\alpha + \delta) m\big)\Big) - \frac{1}{m}\log \Big(1-F_m\big( u_{n,m }\big)\Big)\right\}&=-  I(r_\alpha + \delta) + \frac{1}{\alpha}<0,
\end{flalign*}
we obtain that $u_{n,m } \leq (r_\alpha + \delta) m$ for sufficiently large $n$. Therefore, for large $n$ we have
\begin{flalign*}
\left|\frac{u_{n,m}}{m} - r_\alpha \right| \leq \delta.
\end{flalign*}
Taking $\delta \to 0$ we conclude that
\begin{flalign}\label{u=rm}
\lim_{n \to \infty}  \frac{u_{n,m}}{m} = r_\alpha.
\end{flalign}

Now we use this limit to conclude that for any $s \in \mathbb R$ we have
\begin{flalign*}
\lim_{n \to \infty }  \sum_{k=1}^{m-1} \frac{(m-1)(m-2)\cdots (m-k)}{(u_{n,m}+s)^k} = \frac{1}{r_\alpha-1}.
\end{flalign*}
Indeed, fix $s \in \mathbb R$. Since $\lim_{n \to \infty} \frac{u_{n,m}+s}{m}=r_\alpha>1$, there exists $\varepsilon \in (0,1)$ such that $(m-1) \leq (1-\varepsilon) (u_{n,m}+s)$ for sufficiently large $n$. In particular, for large $n$ we have
\begin{flalign*}
\mathbbm 1_{\{ 1 \leq k \leq m-1  \}}\frac{(m-1)(m-2)\cdots (m-k)}{(u_{n,m}+s)^k} \leq \frac{(m-1)^k}{(u_{n,m}+s)^k} \leq (1-\varepsilon)^k.
\end{flalign*}
Moreover, for fixed $k \in \mathbb N$ we have
\begin{flalign*}
\lim_{n \to \infty } \frac{\mathbbm 1_{\{ 1 \leq k \leq m-1  \}}(m-1)(m-2)\cdots (m-k)}{(u_{n,m}+s)^k}  = \frac{1}{r_\alpha^k}.
\end{flalign*}
Therefore, by the Dominated Convergence Theorem we conclude that
\begin{flalign*}
	\lim_{n \to \infty }  \sum_{k=1}^{m-1} \frac{(m-1)(m-2)\cdots (m-k)}{(u_{n,m}+s)^k} &=	\lim_{n \to \infty }  \sum_{k=1}^{\infty } \frac{\mathbbm 1_{\{ 1 \leq k \leq m-1  \}}(m-1)(m-2)\cdots (m-k)}{(u_{n,m}+s)^k}\\
	& = \sum_{k=1}^{\infty} \frac{1}{r_\alpha^k}\\
	&=\frac{1}{r_\alpha-1}.
\end{flalign*}
Now observe that
\begin{flalign*}
\sum_{k=0}^{m-1} \frac{(u_{n,m} + t)^k}{k!} = \frac{(u_{n,m} + t)^{m-1}}{(m-1)!} \left( 1 +  \sum_{k=1}^{m-1} \frac{(m-1)(m-2)\cdots (m-k)}{(u_{n,m}+t)^k}  \right)
\end{flalign*}
for every $t \in \mathbb R$ and sufficiently large $n$. Thus, by (\ref{1-F}), for sufficiently large $n$ we have
\begin{flalign*}
\frac{1-F_m(u_{n,m}+s)}{1-F_m(u_{n,m})}&=\frac{e^{-s} \sum_{k=0}^{m-1} \frac{(u_{n,m}+s)^k}{k!}}{ \sum_{k=0}^{m-1} \frac{(u_{n,m})^k}{k!}}\\
&=\frac{e^{-s}(u_{n,m} + s)^{m-1} \left( 1 +  \sum_{k=1}^{m-1} \frac{(m-1)(m-2)\cdots (m-k)}{(u_{n,m}+s)^k}  \right)}{(u_{n,m} )^{m-1} \left( 1 +  \sum_{k=1}^{m-1} \frac{(m-1)(m-2)\cdots (m-k)}{(u_{n,m})^k}  \right)} \\
&=\frac{e^{-s}\left(1 + \frac{s}{u_{n,m}}\right)^{u_{n,m}\left(\frac{m-1}{u_{n,m}}\right) } \left( 1 +  \sum_{k=1}^{m-1} \frac{(m-1)(m-2)\cdots (m-k)}{(u_{n,m}+s)^k}  \right)}{ \left( 1 +  \sum_{k=1}^{m-1} \frac{(m-1)(m-2)\cdots (m-k)}{(u_{n,m})^k}  \right)}.
\end{flalign*}
Therefore, for any $s \in \mathbb R$ we have
\begin{flalign*}
\lim_{n \to \infty} \frac{1-F_m(u_{n,m}+s)}{1-F_m(u_{n,m})}&= e^{-\left(  \frac{r_\alpha-1}{r_\alpha}  \right) s}.
\end{flalign*}
In particular, for $s= \theta \left(1 + \frac{1}{r_\alpha -1}\right)$ we obtain
\begin{flalign*}
	\lim_{n \to \infty} \frac{1-F_m\left(u_{n,m}+\theta \left(1 + \frac{1}{r_\alpha -1}\right) \right)}{1-F_m(u_{n,m})}&= e^{-\theta}.
\end{flalign*}
Since $F_m(u_{n,m})=1-1/n$, we conclude that
\begin{flalign*}
F_m \left(\frac{t_{n,m} + \theta w_{n,m}}{n} \right)=F_m\left(u_{n,m}+\theta \left(1 + \frac{1}{r_\alpha -1}\right) \right) = 1 - \frac{e^{-\theta} + o(1)}{n}.
\end{flalign*}
Last, by (\ref{u=rm}), we conclude that $w_{n,m}=o(t_{n,m})$ and $t_{n,m}=o(|w_{n,m}|^2)$.\\

\noindent \textbf{Case III ($\lim_{n \to \infty} \frac{m}{\log n}=+\infty$):} This regime is even more delicate because now we must understand how $u_{n,m}$ fluctuates around $m$. For that purpose, define $r_{n,m}:=-W_{-1}\big(-e^{-1-\log n/m}\big)$ and $\delta_{n,m}:=r_{n,m}-1$. Since 
$$
\delta_{n,m} - \log(1+ \delta_{n,m})=r_{n,m} -1 -\log r_{n,m}= \frac{\log n}{m}
$$
and $\lim_{n \to \infty} \frac{\log n}{m}=0$, we see that $\lim_{n \to \infty} r_{n,m} =1$ and $\lim_{n \to \infty} \delta_{n,m} = 0$. Moreover, since $\log (1+x) = x - \frac{x^2}{2} + O(x^3)$ when $x \to 0$, we have
\begin{flalign}\label{md2}
m \delta_{n,m}^2 = 2 \log n \,(1+o(1)).
\end{flalign}
Now we use (\ref{md2}) to estimate the right tail of $F_m$ at the point $m(1+c \delta_{n,m})$ for $c>0$. Indeed, let $c>0$. Then,
\begin{flalign}
1- F_m \Big(  m(1+c \delta_{n,m})  \Big) &= e^{-m(1+c \delta_{n,m})} \sum_{j=0}^{m-1} \frac{\big(m(1+c \delta_{n,m})\big)^j}{j!} \nonumber \\
&=e^{-m(1+c \delta_{n,m})} \frac{\big(m(1+c \delta_{n,m})\big)^{m-1}}{(m-1)!} \left( 1 +   \sum_{k=1}^{m-1} \frac{(m-1)\cdots (m-k)}{\big( m (1+c \delta_{n,m}) \big)^k}  \right). \label{quase la}
\end{flalign}
Observe that
\begin{flalign*}
1 +   \sum_{k=1}^{m-1} \frac{(m-1)\cdots (m-k)}{\big( m (1+ c\delta_{n,m}) \big)^k} \leq \sum_{k=0}^{+\infty} \frac{1}{(1+ c\delta_{n,m})^k} = \frac{1+c\delta_{n,m}}{c\delta_{n,m}}.
\end{flalign*}
Therefore, there exists a constant $C_1$, dependent on $c$, such that 
\begin{flalign}\label{C1}
	1 +   \sum_{k=1}^{m-1} \frac{(m-1)\cdots (m-k)}{\big( m (1+ c\delta_{n,m}) \big)^k} \leq  \frac{C_1}{\delta_{n,m}}.
\end{flalign}
Now let $K_n:=\left\lfloor 1/ \delta_{n,m} \right\rfloor$. By (\ref{md2}) we have 
\begin{flalign*}
\lim_{n \to \infty} \frac{K_n^2}{m} =0.
\end{flalign*}
Consequently, for every $k  \in \{ 1,\ldots,K_n\}$ we have $\prod_{j=1}^k \Big(1 - \frac{j}{m}\Big)=1+o(1)$. Hence, for sufficiently large $n$ we have
\begin{flalign*}
1 +   \sum_{k=1}^{m-1} \frac{(m-1)\cdots (m-k)}{\big( m (1+ c\delta_{n,m}) \big)^k} &= 1 + \sum_{k=1}^{m-1} \left( \prod_{j=1}^k \Big(1 - \frac{j}{m}\Big)   \right) \frac{1}{(1+ c\delta_{n,m})^k }\\
& \geq  1 + \sum_{k=1}^{K_n} \left( \prod_{j=1}^k \Big(1 - \frac{j}{m}\Big)   \right) \frac{1}{(1+ c\delta_{n,m})^k }\\
& \geq  \frac{1}{2}\sum_{k=0}^{K_n}  \frac{1}{(1+ c\delta_{n,m})^k }\\
&= \frac{1}{2} \left(  \frac{1-{(1+ c\delta_{n,m}) }^{-(K_n+1)}}{1-{(1+ c\delta_{n,m}) }^{-1}} \right).
\end{flalign*}
Since $\lim_{n \to \infty} K_n \delta_{n,m}=1$, we have $\lim_{n \to \infty}{(1+ c\delta_{n,m}) }^{-(K_n+1)}= e^{-c}$. Moreover, since $1-{(1+ c\delta_{n,m}) }^{-1} \leq c\delta_{n,m}$ we conclude that there exists a constant $C_0$, dependent on $c$, such that
\begin{flalign}\label{C0}
	1 +   \sum_{k=1}^{m-1} \frac{(m-1)\cdots (m-k)}{\big( m (1+ c\delta_{n,m}) \big)^k} \geq  \frac{C_0}{\delta_{n,m}}.
\end{flalign}
By (\ref{C1}) and (\ref{C0}) we conclude that
\begin{flalign*}
\log \left( 1 +   \sum_{k=1}^{m-1} \frac{(m-1)\cdots (m-k)}{\big( m (1+ c\delta_{n,m}) \big)^k}   \right) = - \log \delta_{n,m} +\mathcal O(1).
\end{flalign*}
Hence, after long computations using (\ref{quase la}), Stirling's formula and the above identity, we obtain
\begin{flalign*}
\log \left(1- F_m (  m(1+c \delta_{n,m})  )\right)&= -m(1+c \delta_{n,m}) + (m-1) \log ( m(1+c \delta_{n,m})   ) - \log (m-1)! - \log \delta_{n,m} +\mathcal O(1)\\
&= -m \big( c \delta_{n,m} - \log (1+ c \delta_{n,m})  \big) - \frac{1}{2} \log (m \delta_{n,m}^2) + \mathcal O (1).
\end{flalign*}
Finally, by (\ref{md2}) we see that $\log (m \delta_{n,m}^2)= \mathcal O (\log \log n)=o(\log n)$ and that
\begin{flalign*}
m \big( c \delta_{n,m} - \log (1+ c \delta_{n,m})  \big)=  \frac{c^2 }{2} m \delta_{n,m}^2 (1+o(1)) = c^2 \log n (1+o(1)).
\end{flalign*}
Therefore, for every $c>0$ we have
\begin{flalign}\label{tail exp}
\log \left(1- F_m \Big(  m(1+c \delta_{n,m})  \Big)\right)&= - c^2 \log n +o(\log n).
\end{flalign}

Now let $\varepsilon \in (0,1)$. By (\ref{tail exp}) we have
\begin{flalign*}
\log \left(1- F_m \Big(  m(1+(1-\varepsilon) \delta_{n,m})  \Big)\right)&= - (1-\varepsilon)^2 \log n +o(\log n).
\end{flalign*}
Since $(1-\varepsilon)^2<1$, for sufficiently large $n$ we have
\begin{flalign*}
1- F_m \Big(  m(1+(1-\varepsilon) \delta_{n,m})  \Big) > \frac{1}{n} = 1- F_m(u_{n,m}).
\end{flalign*}
Similarly, since $(1+\varepsilon)^2>1$, for sufficiently large $n$ we have
\begin{flalign*}
	1- F_m \Big(  m(1+(1+\varepsilon) \delta_{n,m})  \Big) < \frac{1}{n}=1- F_m(u_{n,m}).
\end{flalign*}
Therefore, for sufficiently large $n$ we have
\begin{flalign*}
F_m \Big(  m(1+(1-\varepsilon) \delta_{n,m})  \Big) < F_m(u_{n,m}) < F_m \Big(  m(1+(1+\varepsilon) \delta_{n,m})  \Big).
\end{flalign*}
Since $F_m$ is increasing, we conclude that for sufficiently large $n$ we have
\begin{flalign*}
1-\varepsilon\leq \frac{u_{n,m} -m}{m \delta_{n,m}} \leq  1+\varepsilon.
\end{flalign*}
Equivalently,
\begin{flalign}\label{limite fino}
\lim_{n \to \infty} \frac{u_{n,m} -m}{m \delta_{n,m}} = 1.
\end{flalign}

Now let $\theta \in \mathbb R$ and define $s_{n,m}(\theta)= \frac{\theta( 1+\delta_{n,m})}{\delta_{n,m}}$. By (\ref{md2}), we have
\begin{flalign}\label{pikachu}
\lim_{n \to \infty}\frac{s_{n,m}(\theta)}{m\delta_{n,m}}=\lim_{n \to \infty}\frac{\theta(1+ \delta_{n,m})}{m\delta_{n,m}^2}=0.
\end{flalign}
Therefore, by (\ref{limite fino}) and (\ref{pikachu}), we obtain
\begin{flalign}\label{raichu}
\lim_{n \to \infty} \frac{(u_{n,m} + s_{n,m}(\theta)) - m}{m \delta_{n,m}} =\lim_{n \to \infty}  \frac{u_{n,m}  - m}{m \delta_{n,m}} + \lim_{n \to \infty} \frac{  s_{n,m}(\theta)}{m \delta_{n,m}}=1.
\end{flalign}
Hence, by (\ref{md2}) and (\ref{raichu}), for every $\theta \in \mathbb R$ we observe that
\begin{flalign*}
\text{(i) } (u_{n,m} + s_{n,m}(\theta)) > m \text{ for sufficiently large } n;&&
\end{flalign*}
\begin{flalign*}
\text{(ii) } \lim_{n \to \infty}\frac{u_{n,m} + s_{n,m}(\theta)}{m} = 1 + \lim_{n \to \infty} \delta_{n,m}\frac{(u_{n,m} + s_{n,m}(\theta))-m}{m\delta_{n,m}} = 1;&&
\end{flalign*}
\begin{flalign*}
\text{(iii) } \lim_{n \to \infty}\frac{\big((u_{n,m} + s_{n,m}(\theta))-m\big)^2}{m} =\lim_{n \to \infty} m \delta_{n,m}^2 \left(\frac{(u_{n,m} + s_{n,m}(\theta))-m}{m \delta_{n,m}}\right)^2=+\infty.&&
\end{flalign*}
In what follows we use the above observations to analyze $F_m( u_{n,m} +s_{n,m}(\theta))$. We start proving that
\begin{flalign}\label{mew}
1 + \sum_{k=1}^m \frac{(m-1)(m-2)\cdots(m-k)}{(u_{n,m} +s_{n,m}(\theta))^k} = \frac{u_{n,m} +s_{n,m}(\theta)}{u_{n,m} +s_{n,m}(\theta) - m} (1+o(1)).
\end{flalign}
Indeed, for each $\theta \in \mathbb R$ define $q_{n,m}(\theta)= m / (u_{n,m} +s_{n,m}(\theta)) \in (0,1)$ (By observation (i), for fixed $\theta \in \mathbb R$, $q_{n,m}(\theta) \in (0,1)$ for sufficiently large $n$). For each $k \in \mathbb N_0$ define
\begin{flalign*}
a_{n,m,k}:=\begin{cases}
1 & \text{ if } k=0,\\
\prod_{j=1}^k \left(1 - \frac{j}{m}\right) &, \text{ if } k\in \{1,2,\ldots,m-1\},\\
0 & \text{ if } k\geq m.
\end{cases}
\end{flalign*}
Since $1 - \prod_{j=1}^k (1-b_j) \leq \sum_{j=1}^k b_j$, for every $b_1,b_2,\ldots,b_k \in [0,1]$ and every $k \in \mathbb N_0$, we have
\begin{flalign*}
0 \leq 1 - a_{n,m,k} \leq \frac{k(k+1)}{2m}
\end{flalign*}
for every $k \in \mathbb N_0$. Therefore,
\begin{flalign*}
0 &\leq \sum_{k=0}^\infty \big(q_{n,m}(\theta)\big)^k (1-a_{n,m,k})= \left(\sum_{k=0}^\infty \big(q_{n,m}(\theta)\big)^k \right) - \left(\sum_{k=0}^\infty \big(q_{n,m}(\theta)\big)^k a_{n,m,k} \right)\\
&=\frac{1}{1-q_{n,m}(\theta)} - \left( 1 + \sum_{k=1}^m \frac{(m-1)(m-2)\cdots(m-k)}{(u_{n,m} +s_{n,m}(\theta))^k}   \right)\\
& \leq \frac{1}{2m}\sum_{k=0}^\infty k (k+1)\big(q_{n,m}(\theta)\big)^k  = \frac{q_{n,m}(\theta)}{m (1-q_{n,m}(\theta))^3}.
\end{flalign*}
Multiplying the above inequalities by $(1-q_{n,m}(\theta))$ we obtain
\begin{flalign*}
0 \leq 1 - (1-q_{n,m}(\theta)) \left( 1 + \sum_{k=1}^m \frac{(m-1)(m-2)\cdots(m-k)}{(u_{n,m} +s_{n,m}(\theta))^k}   \right) \leq \frac{q_{n,m}(\theta)}{m (1-q_{n,m}(\theta))^2}.
\end{flalign*}
Therefore, by observations (ii) and (iii) we have
\begin{flalign*}
\lim_{n \to \infty} \frac{q_{n,m}(\theta)}{m (1-q_{n,m}(\theta))^2} =\lim_{n \to \infty}  \frac{u_{n,m} + s_{n,m}(\theta)}{m} \lim_{n \to \infty} \frac{m}{\big((u_{n,m} + s_{n,m}(\theta))-m\big)^2}=0.
\end{flalign*}
Hence,
\begin{flalign*}
1&= \lim_{n \to \infty} (1-q_{n,m}(\theta)) \left( 1 + \sum_{k=1}^m \frac{(m-1)(m-2)\cdots(m-k)}{(u_{n,m} +s_{n,m}(\theta))^k}   \right)\\
&=  \lim_{n \to \infty} \frac{(u_{n,m} +s_{n,m}(\theta))-m}{u_{n,m} +s_{n,m}(\theta)} \left( 1 + \sum_{k=1}^m \frac{(m-1)(m-2)\cdots(m-k)}{(u_{n,m} +s_{n,m}(\theta))^k}   \right),
\end{flalign*}
which proves (\ref{mew}). From this conclusion and since $s_{n,m}(0)=0$, we obtain
\begin{flalign*}
\frac{1-F_m(u_{n,m}+s_{n,m}(\theta))}{1-F_m(u_{n,m})} &= e^{-s_{n,m}(\theta)} \left(1 + \frac{s_{n,m}(\theta)}{u_{n,m}}\right)^{m-1} \frac{1 + \sum_{k=1}^m \frac{(m-1)(m-2)\cdots(m-k)}{(u_{n,m} +s_{n,m}(\theta))^k}}{1 + \sum_{k=1}^m \frac{(m-1)(m-2)\cdots(m-k)}{(u_{n,m} +s_{n,m}(0))^k}}\\
&=e^{-s_{n,m}(\theta)} \left(1 + \frac{s_{n,m}(\theta)}{u_{n,m}}\right)^{m-1} \frac{u_{n,m} +s_{n,m}(\theta)}{u_{n,m} } \frac{u_{n,m} - m}{(u_{n,m} -m) +s_{n,m}(\theta)} (1+o(1))\\
&=e^{-s_{n,m}(\theta)} \left(1 + \frac{s_{n,m}(\theta)}{u_{n,m}}\right)^{m-1} (1+o(1)).
\end{flalign*}
By (\ref{md2}), observation (i) and since $s_{n,m}(\theta)=\theta\mathcal O ( \delta_{n,m}^{-1})$ for every $\theta \in \mathbb R$, we have
\begin{flalign*}
 \frac{s_{n,m}(\theta)}{u_{n,m}} =  \theta (1+ \delta_{n,m}) \delta_{n,m} \frac{1}{m \delta_{n,m}^2}\frac{m}{u_{n,m}} =o(1)
\end{flalign*}
and
\begin{flalign*}
\frac{m |s_{n,m}(\theta)|^2}{|u_{n,m}|^2} =\mathcal O \left(\frac{1}{m \delta_{n,m}^2}\right) = \mathcal O \left(\frac{1}{\log n}\right).
\end{flalign*}
Therefore,
\begin{flalign*}
\log \left( \frac{1-F_m(u_{n,m}+s_{n,m}(\theta))}{1-F_m(u_{n,m})}  \right)& = - s_{n,m}(\theta) + (m-1) \log \left(  1 + \frac{s_{n,m}(\theta)}{u_{n,m}}  \right) + o(1)\\
&=- s_{n,m}(\theta) + (m-1) \frac{s_{n,m}(\theta)}{u_{n,m}} + \mathcal O \left( \frac{m |s_{n,m}(\theta)|^2}{|u_{n,m}|^2}  \right)\\
&= - s_{n,m}(\theta) \frac{u_{n,m} - m + 1}{u_{n,m}} + o(1)\\
&= - \theta (1+ \delta_{n,m}) \frac{m}{u_{n,m}} \left( \frac{u_{n,m}-m}{m \delta_{n,m}} + \frac{1}{m \delta_{n,m}^2}  \delta_{n,m}  \right) +o(1)\\
&=-\theta + o(1).
\end{flalign*}
Hence, 
\begin{flalign*}
\frac{1-F_m(u_{n,m}+s_{n,m}(\theta))}{1-F_m(u_{n,m})}= e^{-\theta}+o(1).
\end{flalign*}
From the definitions of $\delta_{n,m}$ and $s_{n,m}(\theta)$, and since $1-F_m(u_{n,m})=1/n$, we conclude that
\begin{flalign*}
	F_m \left(\frac{t_{n,m} + \theta w_{n,m}}{n} \right)=F_m\left(u_{n,m}+\theta \left(1 + \frac{1}{r_{n,m} -1}\right) \right) = 1 - \frac{e^{-\theta} + o(1)}{n}.
\end{flalign*}
Finally, by (\ref{md2}) and (\ref{limite fino}), we obtain
\begin{flalign*}
\frac{w_{n,m}}{t_{n,m}}&=\frac{1+\delta_{n,m}}{\delta_{n,m} u_{n,m}} = (1+\delta_{n,m}) \delta_{n,m} \frac{1}{m \delta_{n,m}^2} \left( \frac{u_{n,m}}{m} \right)^{-1}\\
&=(1+\delta_{n,m}) \delta_{n,m} \frac{1}{m \delta_{n,m}^2} \left( 1 + \delta_{n,m} \frac{u_{n,m} - m}{m \delta_{n,m}} \right)^{-1}=o(1)
\end{flalign*}
and
\begin{flalign*}
	\frac{t_{n,m}}{|w_{n,m}|^2}&=\frac{\delta_{n,m}^2 u_{n,m}}{n(1+\delta_{n,m})^2}\\
	&=\frac{u_{n,m}}{m} \frac{m \delta_{n,m}^2}{ \log n} \frac{1}{(1+\delta_{n,m})^2} \frac{ \log n}{n}\\
	&=\left( 1 + \delta_{n,m} \frac{u_{n,m} - m}{m \delta_{n,m}} \right) \frac{m \delta_{n,m}^2}{ \log n} \frac{1}{(1+\delta_{n,m})^2} \frac{ \log n}{n}=o(1).
\end{flalign*}
\end{proof}

The following result is an immediate consequence of (\ref{tau dist}) and Lemma~\ref{profile}.

\begin{corollary}\label{continuous}
Recall the definitions of $\mathcal G$, $t_{n,m}$, $w_{n,m}$ and $\tau_{n,m}$ given in (\ref{gumbel}), (\ref{tn}), (\ref{wn}) and (\ref{tau}), respectively. For every $\theta \in \mathbb R$, we have
\begin{flalign*}
	\lim_{n \to \infty} \mathbb P \left( \tau_{n,m} \leq  t_{n,m} +\theta w_{n,m} \right)=\mathcal G(\theta).
\end{flalign*}
\end{corollary}

Another consequence of Lemma~\ref{profile} is the following:

\begin{corollary}\label{expectation cont}
Recall the definitions of $\gamma$, $t_{n,m}$, $w_{n,m}$ and $\tau_{n,m}$ given in (\ref{euler}), (\ref{tn}), (\ref{wn}) and (\ref{tau}), respectively. We have
	\begin{flalign*}
		\mathbb E [\tau_{n,m}]= t_{n,m} + \gamma w_{n,m} + o(w_{n,m}).
	\end{flalign*}
\end{corollary}

\begin{proof}
For each $x \in \mathbb R$ define
\begin{flalign*}
G_m(x):=1-F_m(x) \,\,\, \text{ and } \,\,\,f_m(x)= \begin{cases}
e^{-x} \frac{x^{m-1}}{(m-1)!} &, \text{ if } x> 0,\\
0 &, \text{ if } x \leq 0.
\end{cases}
\end{flalign*}
Observe that $\log F_m$ and $\log G_m$ are concave functions. We will use this observation to control some tail probabilities.

Let $\theta \geq 1 $. Observe that 
\begin{flalign*}
\frac{t_{n,m}+ w_{n,m}}{n} \in \left[ \frac{t_{n,m}}{n}\,,\, \frac{t_{n,m}+ \theta w_{n,m}}{n}\right]
\end{flalign*}
and that
\begin{flalign*}
\frac{t_{n,m}+  w_{n,m}}{n} = \left(1 - \frac{1}{\theta}\right) \frac{t_{n,m}}{n} + \frac{1}{\theta} \frac{t_{n,m}+ \theta w_{n,m}}{n}.
\end{flalign*}
By the concavity of $\log G_m$ we have
\begin{flalign*}
\log G_m \left( \frac{t_{n,m}+ w_{n,m}}{n}  \right) \geq \left(1 - \frac{1}{\theta}\right) \log G_m \left( \frac{t_{n,m}}{n} \right) + \frac{1}{\theta} \log G_m \left( \frac{t_{n,m}+ \theta w_{n,m}}{n}  \right).
\end{flalign*}
Multiplying both sides above by $\theta$ and reorganizing the terms, we obtain
\begin{flalign*}
\log G_m\left(\frac{t_{n,m}+ \theta w_{n,m}}{n}\right) - \log G_m\left(\frac{t_{n,m}}{n}\right) \leq \theta \left(  \log G_m\left(\frac{t_{n,m}+  w_{n,m}}{n}\right) - \log G_m\left(\frac{t_{n,m}}{n}\right)   \right).
\end{flalign*}
Consequently,
\begin{flalign*}
\frac{G_m\left(\frac{t_{n,m}+ \theta w_{n,m}}{n}\right)}{G_m\left(\frac{t_{n,m}}{n}\right)} \leq \left( \frac{G_m\left(\frac{t_{n,m}+  w_{n,m}}{n}\right)}{G_m\left(\frac{t_{n,m}}{n}\right)} \right)^\theta.
\end{flalign*}
Recalling the definition of $t_{n,m}$ we see that $ G_m\left(\frac{t_{n,m}}{n}\right)= 1/n$. Therefore,
\begin{flalign*}
	n \,{G_m\left(\frac{t_{n,m}+ \theta w_{n,m}}{n}\right)} \leq \left( n\,{G_m\left(\frac{t_{n,m}+  w_{n,m}}{n}\right)} \right)^\theta.
\end{flalign*}
By Lemma~\ref{profile} we have $G_m((t_{n,m}+w_{n,m})/n)=(e^{-1}+o(1))/n$. Then,
\begin{flalign*}
	n \,{G_m\left(\frac{t_{n,m}+ \theta w_{n,m}}{n}\right)} \leq \left(e^{-1} +o(1) \right)^\theta.
\end{flalign*}
Hence, for every $\theta \geq 1$ and every sufficiently large $n$ we have
\begin{flalign}
\mathbb P \left(\frac{\tau_{n,m} - t_{n,m}}{w_{n,m}} > \theta\right)&= 1 - \left(F_m\left(\frac{t_{n,m}+ \theta w_{n,m}}{n}\right) \right)^n \nonumber \\
& \leq n G_m\left(\frac{t_{n,m}+ \theta w_{n,m}}{n}\right) \nonumber \\
& \leq e^{-\theta/2}. \label{right tail}
\end{flalign}

Now observe that $t_{n,m} > w_{n,m}$ for every sufficiently large $n$. Let $c=(e-1)>0$ and define
\begin{flalign*}
c_{n,m}:= n \left( \log F_m \left(\frac{t_{n,m}}{n}\right) - \log F_m \left(\frac{t_{n,m} - w_{n,m}}{n}\right)     \right).
\end{flalign*}
From the definition of $t_{n,m}$ and by Lemma~\ref{profile} we have
\begin{flalign*}
F_m \left(\frac{t_{n,m}}{n}\right)=1-\frac{1}{n} \,\,\,\text{ and } \,\,\, F_m \left(\frac{t_{n,m} - w_{n,m}}{n}\right)=1-\frac{e+o(1)}{n}.
\end{flalign*}
Therefore, $\lim_{n \to \infty} c_{n,m} = c$.

Let $\theta \geq 1 $. Observe that 
\begin{flalign*}
	\frac{t_{n,m}- w_{n,m}}{n} \in \left[ \frac{t_{n,m}- \theta w_{n,m}}{n}\,,\, \frac{t_{n,m}}{n}\right]
\end{flalign*}
and that
\begin{flalign*}
	\frac{t_{n,m}-  w_{n,m}}{n} = \frac{1}{\theta} \frac{t_{n,m}- \theta w_{n,m}}{n} + \left(1 - \frac{1}{\theta}\right) \frac{t_{n,m}}{n} .
\end{flalign*}
By the concavity of $\log F_m$ we have
\begin{flalign*}
	\log F_m \left( \frac{t_{n,m}- w_{n,m}}{n}  \right) \geq  \frac{1}{\theta} \log F_m \left( \frac{t_{n,m}- \theta w_{n,m}}{n}  \right)+ \left(1 - \frac{1}{\theta}\right) \log F_m \left( \frac{t_{n,m}}{n} \right).
\end{flalign*}
Multiplying both sides above by $\theta$ and reorganizing the terms, we obtain
\begin{flalign*}
\log F_m \left( \frac{t_{n,m}- \theta w_{n,m}}{n}  \right)& \leq \theta \log F_m \left( \frac{t_{n,m}-  w_{n,m}}{n}  \right) - (\theta -1 ) \log F_m \left( \frac{t_{n,m}}{n}  \right)\\
& =\log F_m \left( \frac{t_{n,m}-  w_{n,m}}{n}  \right) - (\theta-1) \frac{c_{n,m}}{n}.
\end{flalign*}
Consequently,
\begin{flalign*}
\left(F_m \left( \frac{t_{n,m}- \theta w_{n,m}}{n}  \right) \right)^n \leq \left(F_m \left( \frac{t_{n,m}- w_{n,m}}{n}  \right) \right)^n\,e^{-(\theta-1) c_{n,m}}
\end{flalign*}
Hence, for every $\theta \geq 1$ and every sufficiently large $n$ we have
\begin{flalign}
	\mathbb P \left(\frac{\tau_{n,m} - t_{n,m}}{w_{n,m}} < - \theta\right)&= \left(F_m\left(\frac{t_{n,m}- \theta w_{n,m}}{n}\right) \right)^n \nonumber \\
	& \leq \left(F_m \left( \frac{t_{n,m}- w_{n,m}}{n}  \right) \right)^n\,e^{-(\theta-1) c_{n,m}} \nonumber \\
	& \leq e^{-c(\theta-1)/2}. \label{left tail}
\end{flalign}
Due to (\ref{right tail}) and (\ref{left tail}), we can apply the Dominated Convergence Theorem to conclude that
\begin{flalign*}
\lim_{n \to +\infty}\mathbb E \left[  \frac{\tau_{n,m} - t_{n,m}}{w_{n,m}}   \right]&= \int_0^{+\infty} \lim_{n \to +\infty} \mathbb P \left( \frac{\tau_{n,m} - t_{n,m}}{w_{n,m}} > \theta   \right) d\theta - \int_0^{+\infty} \lim_{n \to +\infty} \mathbb P \left( \frac{\tau_{n,m} - t_{n,m}}{w_{n,m}} < - \theta   \right) d\theta\\
&=\int_0^{+\infty} (1-\mathcal G (\theta)) d\theta - \int_0^{+\infty} \mathcal G (-\theta) d\theta \\
&= \gamma.
\end{flalign*}
\end{proof}

We can finally prove Theorem~\ref{main result}.

\begin{proof}[Proof of Theorem~\ref{main result}]
Let $\mathcal A_n=\{ (m,\ldots,m)  \} \subset \Omega_{n,m}$. By (\ref{T}) and (\ref{tau}), and since $\{Y_t^{n,m}; t \geq 0\}$ is the poissonization of $\{X_\ell^{n,m}; \ell \in \mathbb N_0  \}$, we can use Lemma~\ref{dePoissonization} together with Corollary~\ref{continuous} to obtain the limit in (\ref{erdos}). Furthermore, by Proposition~\ref{expectations} and Corollary~\ref{expectation cont}, we obtain (\ref{expect}).
\end{proof}

\appendix

\section{Total variation distance to absorption and survival probability}\label{A}

In this appendix we prove the equivalence between analyzing the distribution of the absorption time of an absorbing finite Markov chain with fixed initial state and analyzing the total variation distance between its law and its target stationary distribution.

\begin{proposition}\label{dtv}
Let $\{ X_\ell ; \ell \in \mathbb N_0  \}$ be a discrete-time absorbing Markov chain on a finite state space $\Omega$ with fixed initial state $x_0 \in \Omega$. Let $\mathcal A$ be the set of all absorbing states. Let $T=\min\{ \ell \in \mathbb N_0;  X_\ell \in \mathcal A \}$ be the absorption time of the process. Let $\pi$ be the target stationary measure of $\{ X_\ell ; \ell \in \mathbb N_0  \}$ for which $\mu_\ell \to \pi$ as $\ell \to \infty$. For every $\ell \in \mathbb N_0$ we have
\begin{flalign*}
	\| \mu_\ell - \pi \|_{\text{TV}}  &=\mathbb P(T > \ell).
\end{flalign*}
\end{proposition}

\begin{proof}
Let $\ell \in \mathbb N_0$. Observe that 
\begin{flalign*}
\pi(x)=\begin{cases}
\mathbb P (X_T = x|X_0 =x_0) &, \text{ if } x \in \mathcal A,\\
0 &, \text{ if } x \notin \mathcal A.
\end{cases}
\end{flalign*}
Therefore, if $x \in \mathcal A$ then
\begin{flalign*}
\mu_\ell(x)=\mathbb P (T \leq \ell, X_T=x) \leq \pi(x).
\end{flalign*}
Hence,
\begin{flalign*}
\| \mu_\ell - \pi \|_{\text{TV}} &= \frac{1}{2}\sum_{x \in \Omega} |\mu_\ell(x) - \pi(x)|\\
&= \frac{1}{2}\left[\left( \sum_{x \notin \mathcal A} |\mu_\ell(x) - \pi(x)|  \right) + \left( \sum_{x \in \mathcal A} |\mu_\ell(x) - \pi(x)|  \right)\right]\\
&= \frac{1}{2}\left[\left( \sum_{x \notin \mathcal A} \mu_\ell(x)  \right) + \left( \sum_{x \in \mathcal A}  \big(\pi(x) - \mu_\ell(x)\big)  \right)\right]\\
&= \frac{1}{2}\left[\mathbb P (T>\ell) + \left( 1 - \mathbb P (T\leq \ell)  \right)\right]\\
&=\mathbb P (T>\ell).
\end{flalign*}
\end{proof}

\section{Product law of the continuous-time coupon collector chain}\label{B}

In this appendix we prove that, for fixed $t \geq 0$, the random variables $Y_t^{n,m}(1),\ldots,Y_t^{n,m}(n)$, defined in Section~\ref{main arg}, are independent and identically distributed with probability mass function $p_t$ defined in (\ref{pt}):

\begin{proposition}\label{thinning}
Let $\{Y_t^{n,m}; t \geq 0\}$ be the continuous-time Markov chain defined in Section~\ref{main arg}. For each $t \geq 0$ define $\mu_t^{n,m} = \mathbb P (Y_t^{n,m} \in \cdot)$. For every $t \geq 0$ and every $x \in \Omega_{n,m}$ we have
	\begin{flalign*}
		\mu_t^{n,m} (x) = \prod_{i=1}^{n} \left\{ \mathbbm 1_{\{ x_i <m   \}}\,e^{-t/n} \frac{(t/n)^{x_i}}{x_i!}   + \mathbbm 1_{\{ x_i =m   \}} \left( 1 - \,e^{-t/n} \sum_{j=0}^{m-1}\frac{(t/n)^{j}}{j!} \right) \right\}.
	\end{flalign*}
\end{proposition}

We prove Proposition~\ref{thinning} using one of the most interesting tools of information theory, the relative entropy. For that purpose, let $\mu$ and $\nu$ be probability measures defined on a finite space $\Omega$. The Kullback-Leibler divergence, or relative entropy, of $\mu$ from $\nu$ is defined as
\begin{flalign*}
	D_{\text{KL}} (\mu || \nu) = \sum_{x \in \Omega} \mu(x) \log \frac{\mu(x)}{\nu(x)},
\end{flalign*}
with the conventions that $0 \log \frac{0}{0}=0$, $0 \log \frac{0}{q}=0$ and $p \log \frac{p}{0}=\infty$ for $p,q>0$. 

In what follows, we use the outstanding Yau's relative entropy method~\cite{yau} to characterize the law of the process $\{Y_t^{n,m}; t \geq 0  \}$. To do so, let us consider a continuous-time Markov chain $\{Y_t; t \geq 0\}$ with finite state space $\Omega$. Let $\mu_t:=\mathbb P (Y_t \in \cdot)$. Let $\nu$ be a positive measure on $\Omega$ ($\nu(x)>0$ for every $x \in \Omega$) possibly time dependent, and let $\overline \nu$ be a positive measure on $\Omega$ which is not time dependent. Assume that the infinitesimal generator $\mathfrak L$ of $\{Y_t; t \geq 0\}$ is given by 
\begin{flalign*}
	\mathfrak L f(x)=  \sum_{y \in \Omega} K(x,y) \left( f(y) -f(x) \right)
\end{flalign*}
for every function $f:\Omega \to \mathbb R$ and every $x \in \Omega$. Let $\mathfrak L^{\ast,\nu}$ be the adjoint of $\mathfrak L$ with respect to the inner product defined by $\< f,g\>_{\nu}=\int f \cdot g\, d\nu$ for every $f,g: \Omega \to \mathbb R$. A simple computation shows that
\begin{flalign}\label{adjoint}
	\mathfrak L^{\ast,\nu} f (x) = \sum_{y \in \Omega} \left\{ K(y,x) \,f(y)\,\frac{\nu(y)}{\nu(x)} - K(x,y)\,f(x)  \right\}.
\end{flalign}
Let $f_t$ be the Radon–Nikodym derivative of $\mu_t$ with respect to $\nu$. Let $\psi_t$ be the Radon–Nikodym derivative of $\nu$ with respect to $\overline \nu$. We use the following result, that can be found in~\cite{main lemma}, to prove Proposition~\ref{thinning}:

\begin{proposition}[Yau's inequality]\label{yau}
	For every $t > 0$,
	\begin{flalign*}
		\frac{d}{dt} 	D_{\text{KL}} (\mu_t || \nu) \leq 
		\mathbb E \left[ \mathfrak L^{\ast,\nu} \mathbbm 1 (Y_t) - \frac{d}{dt} \log \psi_t (Y_t) \right].
	\end{flalign*}
\end{proposition}

\begin{proof}[Proof of Proposition~\ref{thinning}]
	Let $\mu_t^{n,m}=\mathbb P (Y_t^{n,m} \in \cdot)$. For each $\lambda >0$ let $\nu^{n,m,\lambda}$ be the product measure on $\Omega_{n,m}$ defined by
	\begin{flalign*}
		\nu^{n,m,\lambda} ({ x}):=\prod_{i=1}^{n} \left\{ \mathbbm 1_{\{ x_i <m   \}}\,e^{-\lambda} \frac{\lambda^{x_i}}{x_i!}   + \mathbbm 1_{\{ x_i =m   \}} \left( 1 - \,e^{-\lambda} \sum_{j=0}^{m-1}\frac{\lambda^{j}}{j!} \right) \right\},\,\,\forall\, {x}=(x_1,\ldots,x_n) \in \Omega_{n,m}.
	\end{flalign*}
	Let $\nu_0^{n,m}$ be the Dirac measure on concentrated in the null vector $(0,\ldots,0) \in \Omega_{n,m}$. For each $t > 0$ define $\nu_t^{n,m}:=\nu^{n,m,(t/n)}$. Let $\mathfrak L_{n,m}^{\ast,\nu_t^{n,m}}$ be the adjoint of $\mathfrak L_{n,m}$ with respect to the inner product defined by $\< f,g\>_{\nu_t^{n,m}}=\int f \cdot g\, d\nu_t^{n,m}$ for every $f,g: \Omega_{n,m} \to \mathbb R$.
	Let $f_t^{n,m}$ be the Radon–Nikodym derivative of $\mu_t^{n,m}$ with respect to $\nu_t^{n,m}$. Let $\psi_t^{n,m}$ be the Radon–Nikodym derivative of $\nu_t^{n,m}$ with respect to $\nu^{n,m,1}$. By Proposition~\ref{yau}, for every $t > 0$ we have
	\begin{flalign}
		\frac{d}{dt} 	D_{\text{KL}} (\mu_t^{n,m} || \nu_t^{n,m}) &\leq  \mathbb E \left[ \mathfrak L_{n,m}^{\ast,\nu_t^{n,m}} \mathbbm 1 (Y_t^{n,m}) - \frac{d}{dt} \log \psi_t^{n,m} (Y_t^{n,m}) \right]  \label{yau 1}
	\end{flalign}
	
	To compute the expectation on the right-hand side above, for each ${ x}=(x_1,\ldots,x_n) \in \Omega_{n,m}$ let ${ x}^{i,-}\in \mathbb R^n$ be the vector obtained from ${ x}$ after subtracting $1$ in the $i$-th coordinate, that is,
	\begin{flalign*}
		x^{i,-}_k=\begin{cases}
			x_k - 1 &, \text{ if } k=i,\\
			x_k &, \text{ if } k \in \{1,2,\ldots,n\}\setminus\{i\}.
		\end{cases}
	\end{flalign*}
	Let ${ x}=(x_1,\ldots,x_n) \in \Omega_{n,m}$. By identity (\ref{adjoint}), for every $t>0$ we have
	\begin{flalign}
		\mathfrak L_{n,m}^{\ast,\nu_t^{n,m}} \mathbbm 1 ({ x}) 
		&= \frac{1}{n} \sum_{i=1}^n \left\{  \mathbbm 1_{\{  x_i > 0  \}} \,\frac{\nu_t^{n,m}({ x}^{i,-})}{\nu_t^{n,m}({ x})} - \mathbbm 1_{\{ x_i < m   \}} \right\} \nonumber \\
		&= \frac{1}{n} \sum_{i=1}^n \left\{  \mathbbm 1_{\{  0<x_i <m \}} \,\frac{\nu_t^{n,m}({ x}^{i,-})}{\nu_t^{n,m}({x})} + \mathbbm 1_{\{  x_i =m  \}} \,\frac{\nu_t^{n,m}({ x}^{i,-})}{\nu_t^{n,m}({ x})} - \mathbbm 1_{\{ x_i < m   \}} \right\} \nonumber  \\
		&= \frac{1}{n} \sum_{i=1}^n \left\{  \mathbbm 1_{\{  0<x_i <m \}} \,\frac{e^{-t/n}\frac{(t/n)^{x_i-1}}{(x_i-1)!}}{e^{-t/n}\frac{(t/n)^{x_i}}{x_i!}} + \mathbbm 1_{\{  x_i =m  \}} \,\frac{e^{-t/n}\frac{(t/n)^{m-1}}{(m-1)!}}{1 - \,e^{-t/n} \sum_{j=0}^{m-1}\frac{(t/n)^{j}}{j!}} - \mathbbm 1_{\{ x_i < m   \}} \right\}  \nonumber \\
		&= \frac{1}{n} \sum_{i=1}^n \left\{  \mathbbm 1_{\{  0<x_i <m \}} \,\frac{x_i}{(t/n)} + \mathbbm 1_{\{  x_i =m  \}} \,\frac{e^{-t/n}\frac{(t/n)^{m-1}}{(m-1)!}}{1 - \,e^{-t/n} \sum_{j=0}^{m-1}\frac{(t/n)^{j}}{j!}} - \mathbbm 1_{\{ x_i < m   \}} \right\} \nonumber  \\
		&=  \sum_{i=1}^n \mathbbm 1_{\{ x_i <m   \}}  \left(  \frac{x_i}{t} - \frac{1}{n} \right)  + \mathbbm 1_{\{ x_i =m   \}}  \left( \frac{e^{-t/n}\,(t/n)^{m-1}}{n\,(m-1)!\,\left(1 - \,e^{-t/n} \sum_{j=0}^{m-1}\frac{(t/n)^{j}}{j!}\right)} \right). \label{adjoint1}
	\end{flalign}
	Furthermore, observe that for every $t > 0$ we have
	\begin{flalign}
		\frac{d}{dt} \log \psi_t^{n,m} ({ x})& = \frac{d}{dt} \log \,\prod_{i=1}^{n}  \left\{ \mathbbm 1_{\{ x_i <m   \}}\,e^{-t/n} \frac{(t/n)^{x_i}}{x_i!}   + \mathbbm 1_{\{ x_i =m   \}} \,\left( 1 - \,e^{-t/n} \sum_{j=0}^{m-1}\frac{(t/n)^{j}}{j!} \right) \right\} \nonumber \\
		&=\frac{d}{dt}  \,\sum_{i=1}^{n} \log  \left\{ \mathbbm 1_{\{ x_i <m   \}}\,e^{-t/n} \frac{(t/n)^{x_i}}{x_i!}   + \mathbbm 1_{\{ x_i =m   \}} \,\left( 1 - \,e^{-t/n} \sum_{j=0}^{m-1}\frac{(t/n)^{j}}{j!} \right) \right\} \nonumber  \\
		&=\frac{d}{dt}  \,\sum_{i=1}^{n}   \mathbbm 1_{\{ x_i <m   \}} \log \left(  e^{-t/n} \frac{(t/n)^{x_i}}{x_i!} \right)  + \mathbbm 1_{\{ x_i =m   \}} \log \left( 1 - \,e^{-t/n} \sum_{j=0}^{m-1}\frac{(t/n)^{j}}{j!} \right) \nonumber \\
		&= \frac{d}{dt}  \,\sum_{i=1}^{n}   \mathbbm 1_{\{ x_i <m   \}}  \left( - \frac{t}{n} +  x_i \log \left(\frac{t}{n}\right) - \log(x_i!) \right)  + \mathbbm 1_{\{ x_i =m   \}} \log \left( 1 - \,e^{-t/n} \sum_{j=0}^{m-1}\frac{(t/n)^{j}}{j!} \right) \nonumber \\
		&= \sum_{i=1}^{n}   \mathbbm 1_{\{ x_i <m   \}}  \left(  \frac{x_i}{t} - \frac{1}{n} \right)  + \mathbbm 1_{\{ x_i =m   \}} \left( \frac{e^{-t/n}\,(t/n)^{m-1}}{n\,(m-1)!\,\left(1 - \,e^{-t/n} \sum_{j=0}^{m-1}\frac{(t/n)^{j}}{j!}\right)} \right).\label{log psi}
	\end{flalign}
	Therefore, by (\ref{adjoint1}) and (\ref{log psi}), $\mathfrak L_{n,m}^{\ast,\nu_t^{n,m}} \mathbbm 1 ({ x})- \frac{d}{dt} \log \psi_t^{n,m} ({ x}) =0$ for every ${ x} \in \Omega_{n,m}$. Hence, by inequality (\ref{yau 1}) we conclude that 
	\begin{flalign*}
		\frac{d}{dt} 	D_{\text{KL}} (\mu_t^{n,m} || \nu_t^{n,m}) &\leq 0,\, \,\forall t > 0.
	\end{flalign*} 
	Since $D_{\text{KL}} (\mu_0^{n,m} || \nu_0^{n,m})=0$, we conclude that $	D_{\text{KL}} (\mu_t^{n,m} || \nu_t^{n,m})=0$ for every $t \geq 0$. Hence, $\mu_t^{n,m}= \nu_t^{n,m}$. 
\end{proof}

\vspace{1cm}

\noindent 
\textbf{Acknowledgments.} M.E.C. thanks the Conselho Nacional de Desenvolvimento Científico e Tecnológico (CNPq) through the Programa Institucional de Bolsas de Iniciação Científica (PIBIC), grant 185360/2026-2, and the Fundação de Amparo à Pesquisa do Estado do Rio Grande do Sul (FAPERGS) for the financial support through the Programa Institucional de Bolsas de Iniciação Científica (PROBIC), grant 25/2551-0000730-5. R.M. thanks the Fundação de Amparo à Pesquisa do Estado do Rio Grande do Sul (FAPERGS) for the financial support through the Auxílio Recém Contratado (ARC), grant 25/2551-0000934-0. The authors thank Eldon Barros, Milton Jara and Otávio Menezes, for their availability in many discussions while this project was being developed. The authors also thank Vassilis G. Papanicolaou for bringing Ivchenko's papers and the connection with delay analysis in broadcast channels
to our attention.\\

\noindent \textbf{Statement on AI use.} Chat GPT was used as an auxiliary editorial tool for proof reading, computation checking, figure generator, and literature research assistance. The ideas used in this paper to study the problem were not generated by AI.


\begin{thebibliography}{99}



\bibitem{erdos} Erdős, P., and Rényi, A. \textit{On a classical problem of probability theory.} A Magyar Tudományos Akadémia Matematikai Kutató Intézetének Közleményei 6.1-2 (1961): 215--220.

\bibitem{holst} Holst, L. \textit{On birthday, collectors', occupancy and other classical urn problems.} International Statistical Review/Revue Internationale de Statistique (1986): 15--27.

\bibitem{ilienko} Ilienko, A. \textit{Convergence of point processes associated with coupon collector’s and Dixie cup problems.} (2019): 1--9.

\bibitem{ivchenko1} Ivchenko, G. I. \textit{Limit theorems in an occupancy problem.} Theory of Probability \& Its Applications 16.2 (1971): 293--307.

\bibitem{ivchenko2} Ivchenko, G. I. \textit{On some distributions connected with waiting time in a polynomial scheme.} Theory of Probability \& Its Applications 20.3 (1976): 545--559.

\bibitem{main lemma} Jara, M., and Menezes, O. \textit{Non-equilibrium fluctuations of interacting particle systems.} arXiv preprint arXiv:1810.09526 (2018).

\bibitem{laplace} Laplace, P. S. \textit{Théorie analytique des probabilités.} Vol. 7. Paris: Courcier, 1820.

\bibitem{peres} Levin, D.A., and Peres,Y. \textit{Markov chains and mixing times.} American Mathematical Society, 2026.

\bibitem{long} Long, C. D. \textit{Terminal Defects, Growing Multiplicity, and Variance Extremality in the Double Dixie Cup Problem.} arXiv preprint arXiv:2604.25108 (2026).


\bibitem{newman} Newman, D. J., and Shepp, L. \textit{The double dixie cup problem.} The American Mathematical Monthly 67.1 (1960): 58--61.

\bibitem{papanicolaou} Papanicolaou, V. G., and Doumas, A. V. \textit{On an Old Question of Erd\H {o} s and R\'{e} nyi Arising in the Delay Analysis of Broadcast Channels.} arXiv preprint arXiv:2003.13132 (2020).

\bibitem{sharif1} Sharif, M., and Hassibi, B. \textit{A comparison of time-sharing, DPC, and beamforming for MIMO broadcast channels with many users.} IEEE Transactions on Communications 55.1 (2007): 11--15.

\bibitem{sharif2} Sharif, M., and Hassibi, B. \textit{Delay considerations for opportunistic scheduling in broadcast fading channels.} IEEE Transactions on Wireless Communications 6.9 (2007): 3353--3363.

\bibitem{yau} Yau, H. T. \textit{Relative entropy and hydrodynamics of Ginzburg-Landau models.} Letters in Mathematical Physics, Volume  22 no. 1 (1991), 63--80.

\end{thebibliography}
\end{document}